\documentclass[journal]{IEEEtran}

\ifCLASSINFOpdf

\else

\fi

\usepackage{graphicx}          
\usepackage{amsmath}
\usepackage{amsfonts}
\usepackage{amssymb}
\usepackage{amsthm}
\usepackage{epstopdf}
\usepackage{cite}
\usepackage[colorlinks,citecolor=red,urlcolor=red]{hyperref}
\usepackage{algorithm}
\usepackage{algpseudocode}
\usepackage{booktabs}
\usepackage{multirow}

\DeclareMathOperator*{\argmin}{argmin}
\newtheorem{theorem}{\indent Theorem}
\newtheorem{remark}{\indent Remark}

\begin{document}
%
\title{Outlier-robust Autocovariance Least Square Estimation via Iteratively Reweighted Least Square}
%
%

\author{Jiahong~Li,~\IEEEmembership{Member,~IEEE,}
        Fang~Deng,~\IEEEmembership{Senior Member,~IEEE}
\thanks{J.~Li is with the College of Robotics, Beijing Union University, Beijing 100101, China (email: jqrjiahong@buu.edu.cn).
F.~Deng is with the Key Laboratory of Intelligent Control and Decision of Complex Systems, and also with the School of Automation, Beijing Institute of Technology, Beijing 100081, China (email: dengfang@bit.edu.cn). \emph{Corresponding author: Fang Deng} Code: https://github.com/jiahongljh/als-irls}}


%
%

\markboth{IEEE,~Vol.~, No.~, ~}%
{Shell \MakeLowercase{\textit{et al.}}: Bare Demo of IEEEtran.cls for IEEE Journals}
%



\maketitle

\begin{abstract}
The autocovariance least squares (ALS) method is a computationally efficient approach for estimating noise covariances in Kalman filters without requiring specific noise models. However, conventional ALS and its variants rely on the classic least mean squares (LMS) criterion, making them highly sensitive to measurement outliers and prone to severe performance degradation. To overcome this limitation, this paper proposes a novel outlier-robust ALS algorithm, termed ALS-IRLS, based on the iteratively reweighted least squares (IRLS) framework. Specifically, the proposed approach introduces a two-tier robustification strategy. First, an innovation-level adaptive thresholding mechanism is employed to filter out heavily contaminated data. Second, the outlier-contaminated autocovariance is formulated using an $\epsilon$-contamination model, where the standard LMS criterion is replaced by the Huber cost function. The IRLS method is then utilized to iteratively adjust data weights based on estimation deviations, effectively mitigating the influence of residual outliers. Comparative simulations demonstrate that ALS-IRLS reduces the root-mean-square error (RMSE) of noise covariance estimates by over two orders of magnitude compared to standard ALS. Furthermore, it significantly enhances downstream state estimation accuracy, outperforming existing outlier-robust Kalman filters and achieving performance nearly equivalent to the ideal Oracle lower bound in the presence of noisy and anomalous data.\end{abstract}


%
\IEEEpeerreviewmaketitle

\section{Introduction}\label{sec:intro}
For the state estimation problem of linear time-invariant (LTI) systems with zero-mean Gaussian process noise and measurement noise, the {K}alman filter (KF) is proven to be optimal in terms of the minimum mean square error (MMSE) criterion\cite{Kalman1960}. 
However, the statistics of the noise covariance are often unknown or only partially known due to mismatches in the nominal system or the invalidity of offline calibration in practice, which jeopardizes the filtering performance \cite{TerryMoore2003, li2018f, Deng2017}.

Existing noise covariance estimation methods utilize historical open-loop data, which can be classified as Bayesian\cite{Lainiotis1971, Sarkka2009}, maximum likelihood estimation (MLE) \cite{Kashyap1970, zagrobelny2015i}, covariance matching (CM) \cite{Myers1976, solonen2014e}, minimax \cite{Tugnait1979}, subspace identification (SI) \cite{van2012s, Ljung1998s}, correlation \cite{Mehra1970, Belanger1974, Odelson2006, Rajamani2009, Dunik2016, arnold2021t}.
The Bayesian method offers optimal posterior estimation by using Bayes' rule to update the prior knowledge of the noise covariance in online or multi-model scenarios but suffers from computational complexity and sensitivity to hyper-parameter choices \cite{Sarkka2009}.
The MLE method maximizes the likelihood of innovations without prior knowledge but is computationally intensive and sensitive to model misspecification \cite{zagrobelny2015i}. 
The CM method, a robust and computationally efficient approach by minimizing the difference between the estimated noise covariance and the true one, is sensitive to outliers and unsuitable for online scenarios \cite{meng2016c}. 
The minimax method provides an optimal estimate that minimizes the worst-case error, which is robust against model uncertainty but results in conservative estimates and high computational demand \cite{Verdu1984}.
The SI method, derived from solving the parametric optimization problem involving projections of the Hankel matrix from input-output data, is an efficient approach but imposes stringent model requirements and sensitivity to model misspecification \cite{Ljung1998s}. 
Among these methods, correlation methods have attracted considerable attention due to reasonable computational complexity and no specific noise model requirements \cite{dunik2017n}.
It was initially proposed by Mehra and B\'{e}langer as a three-step procedure \cite{Mehra1970, Belanger1974}, and later streamlined into the auto-covariance least-squares (ALS) method as a single-step procedure \cite{Odelson2006, Rajamani2009, Dunik2016,arnold2021t}. 
However, the above methods don't account for the influence of measurement outliers and innovation outliers that often occur due to sensor faults or external interfaces in practical systems \cite{hammes2009r}.

Outliers present significant challenges to the estimation accuracy of traditional KF methods since the outliers corrupt the Gaussian noise, resulting in bias, high variance, or even breakdown\cite{huber2004r,zoubir2012r,maronna2019r}.
To mitigate the effect, various types of outlier-robust {K}alman filters (ORKF) have been proposed to guarantee that the residuals should remain bounded and Gaussian to resist outliers' effect.
Examples of ORKF include parametric methods that use conjugate prior distribution \cite{wang2018r,huang2020n,zhu2021n} or Student-t distribution \cite{huang2019n,zhu2022s} and non-parametric methods that use informational quantities (e.g., correntropy) \cite{wang2020o,li2023i} or robust regression technique called M-estimator \cite{gandhi2009r,ting2007k,thanoon2015r,miller2020r,hampel2011s}.
The conjugate prior distribution method approximates the posterior probability density function (pdf) of the state contaminated by outliers by using the conjugate prior distribution, such as Beta-Bernoulli (BB) \cite{wang2018r}, Normal-inverse-Wishart (NIW) \cite{huang2020n} or Normal-inverse-gamma (NIG) distribution \cite{zhu2021n}.
The Student-t method approximates the heavy-tailed state pdf induced by outliers as a Student-t pdf with fixed degrees of freedom (DOF) based on the Bayesian rule, or a Gaussian pdf via the variational Bayesian (VB) approach \cite{zhu2022s}.
The parametric methods are effective with proper settings but may perform badly when inaccurate or even erroneous distribution families are used.
Besides, these methods are hard to implement in real-time applications due to high computational complexity.
The nonparametric methods are more flexible because no specification of a prior distribution is required. 
The maximum correntropy criterion (MCC) method handles state and measurement outliers by choosing the sum of Gaussian kernel functions of the prediction error and residual error as the robust cost function \cite{wang2020o}. 
The M-estimator mitigates the influence function (IF) of deviant observations by minimizing a robust cost function instead of the least mean square (LMS) function, e.g., the weighted least mean square (WLMS) function \cite{ting2007k,chang2017u} that assigns less influence on outliers, the least absolute deviation (LAD) function \cite{thanoon2015r,miller2020r} that utilizes a $\ell_1$ norm, the Huber function \cite{huber2004r,chang2017u} and its variants \cite{hampel2011s,yu2019m} that utilize a combined $\ell_1$ and $\ell_2$ norm.
However, these aforementioned methods have a fixed influence function determined by their loss function and tuning parameters, which may not adapt well to the specific outlier structure.

To develop a more robust and effective M-estimator, the iteratively reweighted least squares (IRLS) algorithm was developed \cite{holland1977r} and broadly applied in rotation averaging \cite{hartley2011l1,chatterjee2017r}, triangulation\cite{navidi1998s,bartoli2005s}, point cloud alignment\cite{bergstrom2014r}. 
IRLS iteratively minimizes a robust cost function and adjusts the weighting for each data based on the error by solving the weighted least squares problems iteratively, thereby mitigating the influence of outliers and increasing the estimation accuracy \cite{daubechies2010i}.
By iteratively reweighting observations, IRLS provides the flexibility to use a variety of norms, making it a more versatile tool to implement and compatible with other weight functions\cite{marx2013r}, especially for nonlinear regression \cite{tao2023o}.Besides, IRLS exhibits linear and even sublinear convergence properties under certain conditions\cite{aftab2015c,mukhoty2019g,peng2022g}.

This paper focuses on the outlier-robust noise covariance estimation problem. 
To address the sensitivity issue to outliers, the outlier-robust ALS algorithm is developed based on the IRLS algorithm. 
Our contributions are threefold:
\begin{enumerate}
    \item We establish a connection between the outlier-robust regression and the ALS-based noise covariance estimation by replacing the original LMS criterion ($\ell_2$-norm) with the Huber criterion ($\ell_p$-norm), reformulating the outlier-robust noise covariance estimation problem as robust regression problem of the ALS method.
    \item We develop an outlier-robust ALS algorithm termed ALS-IRLS via the IRLS algorithm. IRLS adjusts the weighting for each data based on the Huber-based cost function by iteratively solving the weighted least squares problems, thereby mitigating the influence of outliers and increasing the estimation accuracy.
    \item We run the proposed ALS-IRLS algorithm and other ORKF algorithms on the comparative simulations. The simulation results demonstrate the superiority of the proposed ALS-IRLS algorithm over others in terms of accuracy and robustness to outliers. 
\end{enumerate}

The rest of the paper is organized as follows.
Section \ref{sec:problem} introduces the ALS method for the noise covariance estimation problem via a linear regression approach, and formulate the outlier-robust noise covariance estimation problem which requires a robust regression solution.
In Section \ref{sec:ORALS}, A novel outlier-robust algorithm, termed as ALS-IRLS, which utilizes the IRLS approach to the robust regression is developed to address the problem. 
Numerical simulations of the proposed algorithm are discussed in Section \ref{sec:simulation}.
The simulation results demonstrate that the ALS-IRLS algorithm outperforms others in terms of accuracy and robustness to outliers.
Section \ref{sec:conclusions} provides the conclusion.
\section{preliminaries and problem formulation}
\label{sec:problem}
Consider a linear time-invariant (LTI) dynamic system with an LTI measurement model as the state-space model (SSM):
\begin{equation}\label{LTI}
\begin{aligned}
  x_{k+1} &= Fx_{k}+w_{k} \quad k\in \mathbb{N}^{+}\\
  z_{k} &= Hx_{k} + v_{k}, \quad i\in \mathcal {V}
  \end{aligned}
\end{equation}
where the vector $x_{k}\in \mathbb{R}^{n_{x}}$ and $z_{k}\in \mathbb{R}^{n_{z}}$ represent the state and the measurement at time instant $k\in \mathbb{N}^{+}$, $\mathbb{R}^{n}$ denotes the $n$-dimensional real space.
$F\in \mathbb{R}^{n_{x}\times n_{x}}$ and $H\in \mathbb{R}^{n_{z}\times n_{x}}$ are state transitional matrix and measurement transitional matrix.
$w_{k}$ and $v_{k}$ represent the process noise and the measurement noise, respectively, and are mutually independent following the zero-mean Gaussian statistics with probability $w_{k}\sim \mathcal{N}(0_{n_{x}\times 1}, Q)$ and $v_{k}\sim \mathcal{N}(0_{n_{z}\times 1}, R)$. $Q\in \mathbb{R}^{n_{x}\times n_{x}}$ and $R\in \mathbb{R}^{n_{z}\times n_{z}}$ denote the model and the noise covariance matrix of the measurement.

When the noise covariance matrix $Q$ and $R$ are known, the optimal estimates of the state $\hat{x}_{k}$ and the covariance $P_{k}$ at time $k$ can be derived via {K}alman filter (KF) in terms of minimum mean square error (MMSE):
\begin{equation}\label{KF}
\begin{aligned}
&\begin{array}{l}
\hat{x}_{k|k-1} = F\hat{x}_{k-1|k-1}\\
P_{k|k-1} = FP_{k-1|k-1}F^\mathrm{T}+Q 
\end{array} \\
&\begin{array}{l}
\hat{x}_{k|k}=\hat{x}_{k|k-1}+ K(z_{k}-H\hat{x}_{k|k-1})\\
K =P_{k|k-1}H^\mathrm{T}(HP_{k|k-1}H^\mathrm{T}+R)^{-1}\\
P_{k|k} = (I-KH)P_{k|k-1}
\end{array}
\end{aligned}
\end{equation}
where $K\in\mathbb{R}^{n_{x}\times n_{z}}$ is denoted as the {K}alman gain.

However, in practice, $Q$ and $R$ are often unknown, and inaccurate settings degrade filter performance.
Therefore, it is necessary to calculate the exact values of $Q$ and $R$ accurately.

The auto-covariance least-squares estimation (ALS) method is designed for the estimation of both the process noise covariance $Q$ and the measurement noise covariance $R$ \cite{Odelson2006}.
It relies on the autocovariance of the innovation sequence $e_{k}= z_{k}-H\hat{x}_{k|k-1}$ from a linear state predictor $\hat{x}_{k+1|k}=F\left(\hat{x}_{k|k-1}+K e_k\right)$ for $\forall k$, which is defined as the expectation of the data with its lagged one, denoted as $C_{e,j} \triangleq \mathbb{E}[e_{k+j}e_{k}^\mathrm{T}]$.
Denote the residual as $\varepsilon_{k}=x_{k}-\hat{x}_{k|k-1}$, then it satisfies:
\begin{equation}\label{varepsilon}
  \varepsilon_{k+1} = \underbrace{(F-FKH)}_{\bar{F}}\varepsilon_{k} + \underbrace{[I_{n_{x}},-FK]}_{G}\underbrace{\left[
                                  \begin{array}{c}
                                    w_{k} \\
                                    v_{k} \\
                                  \end{array}
                                \right]}_{\bar{w}_{k}}
\end{equation}
where $\varepsilon_{k}$ is related to the innovations $e_{k}$ by
\begin{equation}
  e_k=H\left(x_k-\hat{x}_{k|k-1}\right)+v_k=H \varepsilon_k+v_k
\end{equation}

The {K}alman gain of (\ref{varepsilon}) converges to a steady-state gain exponentially fast, which can be calculated offline, i.e., $K=P_\epsilon H^\mathrm{T}\left(H P_\epsilon H^\mathrm{T}+R\right)^{-1}$, where $P_\epsilon$ denotes the steady-state residual covariance matrix following the Lyapunov equation in (\ref{Pest}) exists given that the prediction error transitional matrix $\bar{F}$ is stable and the innovation sequence is a stationary process.
\begin{equation}\label{Pest}
  P_{\varepsilon} = \bar{F}P_{\varepsilon}\bar{F}^\mathrm{T} + G\Sigma G^\mathrm{T}
\end{equation}
where $\Sigma = \left[
         \begin{array}{cc}
           Q & 0_{n_x\times n_z} \\
           0_{n_z\times n_x} & R \\
         \end{array}
       \right]$. 

Then the auto-covariance of the innovation sequence is:
\begin{equation}\label{AC_innovations}
\begin{aligned}
  C_{e,0} &\triangleq \mathbb{E}[e_{k}e_{k}^\mathrm{T}] = H P_\varepsilon H^\mathrm{T}+R \\
  C_{e,j} &\triangleq \mathbb{E}[e_{k+j}e_{k}^\mathrm{T}] = H\bar{F}^j P_\epsilon H^\mathrm{T}-H\bar{F}^{j-1}FKR
\end{aligned}
\end{equation}
where $j=1,2,\ldots,N-1$, and $N$ is a user-defined parameter defining the window size.
By substituting the solution to (\ref{Pest}) into (\ref{AC_innovations}), the noise covariance estimation problem can be reformulated into a linear regression problem in (\ref{Ab}) \cite{Odelson2006}.
\begin{equation}\label{Ab}
  \mathcal{A}\theta = \mathbf{b}
\end{equation}
where the matrix $\mathcal{A}$ is defined as
\begin{equation}\label{A_LS}
\begin{aligned}
  \mathcal{A} &= \big[D,D(FK\otimes FK)+(I_{n_{z}}\otimes \Gamma)\big] \\
  D &= \big(H\otimes\mathcal{O}\big)\big(I_{n_{x}^{2}}-\bar{F}\otimes\bar{F}\big)^{-1} \\
  \mathcal{O} &= \big[H^\mathrm{T},(H\bar{F})^\mathrm{T},\ldots,(H\bar{F}^{N-1})\big]^\mathrm{T} \\
  \Gamma &= \big[I_{n_{z}},-(HFK)^\mathrm{T},\ldots, -(H\bar{F}^{N-2}FK)^\mathrm{T}\big]^\mathrm{T}
  \end{aligned}
\end{equation}
where the symbol $\otimes$ stands for the Kronecker product.
Denote the dependent variables as $\theta=[Q_s^\mathrm{T},R_s^\mathrm{T}]^\mathrm{T}$ and $\mathbf{b}=(C_{e}(N))_{s}$ with $C_{e}(N) = [C_{e,0},C_{e,1}^\mathrm{T},\ldots,C_{e,N-1}^\mathrm{T}]^\mathrm{T}$.
$Q_s$ means the column-wise stacking of the matrix $Q$ into a vector.

The parameter $\theta$ can be easily estimated by solving the semi-definite constrained least squares problem via the ordinary least square (OLS) method in terms of LMS criterion:
\begin{equation}\label{ALS}
\begin{aligned}
  \hat{\theta}^{\star} &= \argmin_{\theta}\|\mathbf{b}-\mathcal{A}\theta\|_{2}^2 \\
  & s.t.,\quad Q,R \ge 0
\end{aligned}
\end{equation}
where the optimal estimate $\hat{\theta}^{\star}$ is given in the minimum mean-square error (MMSE) sense if the matrix inequality holds:
\begin{equation}\label{ALS_estimate}
  \hat{\theta}^{\star} = (\mathcal{A}^\mathrm{T}\mathcal{A})^{-1}\mathcal{A}^\mathrm{T}\hat{\mathbf{b}}=\mathcal{A}^{\dag}\hat{\mathbf{b}}
\end{equation}
where $\hat{\mathbf{b}}$ denotes the unbiased estimate of the vector $\mathbf{b}$, and is computed by using the ergodic property of the $N$-innovations to estimate the auto-covariance matrix $C_{e}(N)$:
\begin{equation}\label{bhat}
  \begin{aligned}
  &\hat{\mathbf{b}} = \big(\hat{C}_{e}(N)\big)_{s}= [\hat{C}_{e,0}^\mathrm{T},\hat{C}_{e,1}^\mathrm{T},\ldots,\hat{C}_{e,N-1}^\mathrm{T}]_{s}^\mathrm{T} \\
  &\hat{C}_{e,j} = \frac{1}{\tau-j}\sum_{k=1}^{\tau-j}e_{k+j}e_{k}^{\mathrm{T}},\quad j=0,1,\ldots,N-1
\end{aligned}
\end{equation}
where $\tau$ is denoted as the window size of the auto-covariances.

Due to unexpected maneuvers and unreliable sensors, the correlated measurements may be contaminated by outliers, resulting in substantial performance degradation of the noise covariance estimation via the ALS algorithm in (\ref{ALS}) and the state estimate via KF.
Therefore, it is necessary to address the sensitivity to outliers of the original ALS method. 
By modeling the observation of the autocovariance contaminated with outliers as the $\epsilon$-contamination model $\mathbf{b}=\mathcal{A}\theta+\epsilon$, the original least mean squares (LMS) criterion in (\ref{ALS}) is reformulated as an outlier-robust regression problem:
\begin{equation}\label{OutlierRobustRegression}
\theta^{\star} = \arg\min_{\theta \in \mathbb{R}^p} \sum_{i=1}^{n} \rho(r(\theta))
\end{equation}
where $r(\theta)$ denotes the residual function, i.e., the deviation $\mathbf{b}-\mathcal{A}\theta$ for the noise covariance estimation problem, with $\mathcal{A}_i \in \mathbb{R}^{N \times (n_x+n_z))}$ representing the design matrix, $\theta \in \mathbb{R}^{(n_x+n_z)})$ denotes the noise covariance parameter, and $b_i \in \mathbb{R}^{(n_x+n_z)})$ denotes the samples of the autocovariance of the innovations.
The robust cost function $\rho(\cdot)$ is designed to mitigate the effect of outliers, and is often selected to be the Huber loss function, which is a piecewise function that applies a squared loss for small residuals and a linear loss for large residuals:
\begin{equation}\label{huberloss}
\rho(r)= 
\begin{cases}
\frac{1}{2} r^2, & \text{if } |r| \leq \tau, \\
\tau \left( |r| - \frac{1}{2}\tau \right), & \text{if } |r| > \tau,
\end{cases}
\end{equation}
where $\tau$ is a user-defined threshold that balances bias and robustness. The weight $w$ is based on the chosen $\psi$ function, which is often the derivative of the Huber loss function concerning the deviation $\mathbf{b}-\mathcal{A}\theta$:
\begin{equation}
w = \psi(r) = \frac{\partial \rho(r)}{\partial r}=
\begin{cases}
1, & \text{if } |r| \leq \tau, \\
\frac{\tau}{|r|}, & \text{if } |r| > \tau.
\end{cases}
\end{equation}
where the weight $w$ is adjusted by the absolute value of the residual $r$, which is 1 when $|r| \leq \tau$, and reduced when $|r| \ge \tau$ to mitigate the influence of large outliers.
However, the use of a fixed weight function in the Huber M-estimator can lead to suboptimal performance if the initial residuals are not representative of the underlying data distribution.
To address the limitation, it is necessary to develop a novel algorithm that allows for the adaptation to changes in the data distribution and provides a more accurate estimation in the presence of outliers, as is discussed in Section \ref{sec:ORALS}.
\section{Proposed algorithm}\label{sec:ORALS}
The iteratively reweighted least squares (IRLS) algorithm \cite{marx2013r} is an effective approach to solving certain optimization problems with the objective functions of the $p$-norm in (\ref{OutlierRobustRegressionlp}). 
\begin{equation}\label{OutlierRobustRegressionlp}
\theta^{\star} = \arg\min_{\theta \in \mathbb{R}^p} \sum_{i=1}^{n} \rho(r(\theta)) = \arg\min_{\theta \in \mathbb{R}^p} \sum_{i=1}^{n} \left\|b_i-\mathcal{A}_i\theta\right\|^p
\end{equation}
The IRLS algorithm at step $t+1$ assigns a weight for each data and solves a weighted least squares problem iteratively:
\begin{equation}
\begin{aligned}
\theta^{(t+1)} &= \arg\min_{\theta} \sum_{i=1}^{n} w_i(\theta^{(t)})(b_i - \mathcal{A}_i^T\theta)^2 \\
&=\left(\mathcal{A}^T \mathbf{W}^{(t)} \mathcal{A}\right)^{-1} \mathcal{A}^{\mathrm{T}} \mathbf{W}^{(t)} \mathbf{b}
\end{aligned}
\end{equation}
where $\mathbf{W}^{(t)}$ denotes the diagonal matrix of weights at step $t$, $w_{i}$ denotes its $i^{th}$ diagonal element and is set initially as the unity weighting, i.e., $w_{i}^{(0)}=1$. 
After each iteration, the weight $w_{i}$ is updated to $w_{i}^{t}=\left\|b_i-\mathcal{A}_i\theta^{t}\right\|^{p-2}$.
Specifically, for the Huber loss function in (\ref{huberloss}), the weight $w_{i}$ associated with the observation $b_{i}$ at step $t$ is determined by:
\begin{equation}\label{wi_huber}
w_i^{(t)}= \psi(r_i(\theta^{(t)}))=
\begin{cases}
1 & \text{if } |b_{i}-\mathcal{A}_{i}\theta^{(t)}| \leq \delta, \\
\frac{\delta}{|b_{i}-\mathcal{A}_{i}\theta^{(t)}|} & \text{if } |b_{i}-\mathcal{A}_{i}\theta^{(t)}| > \delta.
\end{cases}
\end{equation}

To guarantee the convergence, the process is repeated until convergence, which is typically defined as the change of the objective function $|\rho(\theta^{(k+1)}) - \rho(\theta^{(k)})|$ or the estimated parameters $|\theta^{(k+1)} - \theta^{(k)}|$ between successive iterations is smaller than a specific threshold $\xi$.
The IRLS algorithm is detailed in Alg. \ref{alg:irls}.
\begin{algorithm}
\caption{Iteratively reweighted least squares (IRLS) Algorithm for outlier-robust regression problem in (\ref{OutlierRobustRegression})}\label{alg:irls}
\begin{algorithmic}[1]
\Require Design matrix $\mathcal{A} \in \mathbb{R}^{N \times (n_x+n_z)}$, observation vector $\mathbf{b} \in \mathbb{R}^{N}$, convergence threshold $\xi$, Huber threshold $\delta$, initial parameter estimate $\theta^{(0)} \in \mathbb{R}^{(n_x+n_z)}$, maximum iterations $T$.
\Ensure Estimated parameter vector $\theta^\star=[Q_s^\star,R_s^\star]^\mathrm{T}$.
\State Initialize $\theta^{(1)} \gets \theta^{(0)}$, $\mathbf{W}^{(0)}$ is an identity matrix $I_N$.
\State Initialize $t \gets 0$
\Repeat
    \State $t \gets t + 1$
    \For{$i \gets 1$ to $n$}
        \State Compute the deviation $r_i^{(t-1)} \gets b_i - \mathcal{A}_i \theta^{(t-1)}$.
        \State Compute the weight $w_i^{(t)}$ in (\ref{wi_huber}).
    \EndFor
    \State Construct diagonal matrix $\mathbf{W}^{(t)}$ with elements $w_i^{(t)}$.
    \State Solve the weighted least squares problem: $\theta^{(t)} = \left(\mathcal{A}^T \mathbf{W}^{(t)} \mathcal{A}\right)^{-1}\mathcal{A}^T \mathbf{W}^{(t)} \mathbf{b}$.
\Until{convergence criterion is met: $|\rho(\theta^{(t)}) - \rho(\theta^{(t-1)})| < \xi$ or $|\theta^{(t)} - \theta^{(t-1)}| < \xi$ or $t \geq T$}
\State $\theta^\star \gets \theta^{(t)}$
\end{algorithmic}
\end{algorithm}

Therefore, by using the IRLS algorithm, the novel outlier-robust ALS algorithm is developed for the outlier-robust noise covariance estimation problem in Alg. \ref{als_alsirls}.
\begin{algorithm}
\caption{Outlier-robust ALS algorithm via IRLS}\label{als_alsirls}
\begin{algorithmic}[1]
\Require the system matrices $F$, $H$, initial estimates $\hat{x}_{0|0}$, initial covariance estimates $\hat{Q}^{(0)}$, $\hat{R}^{(0)}$, window size $N$, maximum iterations for IRLS algorithm $T$, convergence threshold $\xi$, maximum simulation time $N_{sim}$.
\Ensure noise covariance estimates $\hat{Q}$, $\hat{R}$.
\State Initialize $k \gets 0$, $\hat{x}_{1|0} \gets F\hat{x}_{0|0}$, iteration loop $N_k=0$.
\Repeat
\State Compute {K}alman gain $\hat{K}^{(N_k)}$ by solving the steady-state solution $\hat{P}^{(N_k)}$ in (\ref{Pest}) given $\hat{Q}^{(N_k)}$ and $\hat{R}^{(N_k)}$.
\State Construct the designed matrix $ \mathcal{A}$ in (\ref{A_LS}).
\For{$k \gets N\times N_k+1$ to $N\times N_k+N$}
\State Generate the measurements $z_{k}$ according to the LTI system modeled as the SSM in (\ref{LTI}).
  \State Compute the innovation $e_k\gets z_{k}-H\hat{x}_{k|k-1}$.
  \State Compute the linear state predictor $\hat{x}_{k+1|k}=F\left(\hat{x}_{k|k-1}+\hat{K} e_k\right)$ via KF in (\ref{KF}).
\EndFor
  \State Given the innovation sequence $\{e_{N\times N_k+1:N\times N_k+N}\}$, compute the empirical autocovariance estimate vector $\mathbf{b}=[\hat{C}_{e,0}^\mathrm{T},\hat{C}_{e,1}^\mathrm{T},\ldots,\hat{C}_{e, N-1}^\mathrm{T}]_{s}^\mathrm{T}$, where $\hat{C}_{e,j} = \frac{1}{\tau-j}\sum_{k=1}^{\tau-j}e_{k+j}e_{k}^{\mathrm{T}}$ for $j = 0, \ldots, N-1$ in (\ref{bhat}).
  \State $N_k \gets N_k + 1$
  \State Compute the outlier-robust estimate $\theta^{(N_k)}=[Q_s^\star,R_s^\star]^\mathrm{T}$ at $N_k$ loop via IRLS algorithm in Alg. \ref{alg:irls}.
  \State Update $Q^{(N_k)}$ and $R^{(N_k)}$ from $\theta^{(N_k)}$.
\Until{convergence of $\hat{Q}$ and $\hat{R}$ or $k \geq N_{sim}$}
\State $\hat{Q} \gets Q^{(N_k)}$, $\hat{R} \gets R^{(N_k)}$
\end{algorithmic}
\end{algorithm}
\begin{theorem}[Convergence of ALS-IRLS]\label{thm:convergence}
Let $\mathcal{A}^{\mathrm{T}}\mathcal{A}$ be positive definite and
$\delta>0$.
Then the sequence $\{\theta^{(t)}\}$ generated by
Algorithm~\ref{alg:irls} satisfies
$\rho(\theta^{(t+1)})\leq\rho(\theta^{(t)})$ for all $t$, and every
limit point of $\{\theta^{(t)}\}$ is a stationary point of~\eqref{OutlierRobustRegression}.
Since the Huber objective is strictly convex under the stated condition,
the sequence converges to the unique global minimizer $\theta^{\star}$.
\end{theorem}
\begin{proof}
The Huber weight $\psi(r)$ in~\eqref{wi_huber} satisfies
$\psi(r)=\rho'(r)/r>0$ for all $r\neq0$, so $\mathbf{W}^{(t)}$ is
positive definite.
Each IRLS step solves a strictly convex quadratic
surrogate that majorizes $\rho(\cdot)$ at $\theta^{(t)}$, hence
$\rho(\theta^{(t+1)})\leq\rho(\theta^{(t)})$~\cite{daubechies2010i,lai2013i}.
Since $\rho$ is bounded below and every $\theta^{(t)}$ lies in a compact
sublevel set, the sequence converges.
Strict convexity of $\rho$ (when $\mathcal{A}^{\mathrm{T}}\mathcal{A}\succ0$)
uniquely identifies the limit point as $\theta^{\star}$.
\end{proof}

\begin{theorem}[Computational complexity]\label{thm:complexity}
Let $n_x$ and $n_z$ denote the dimensions of the state and measurement,
respectively, and let $T$ be the number of IRLS iterations per outer loop.
The per-outer-iteration complexity of Algorithm~\ref{als_alsirls} is
$\mathcal{O}\!\left(T\,Nn_z^2(n_x^2+n_z^2)^2\right)$.
\end{theorem}
\begin{proof}
In each IRLS step, the dominant operations are:
(i)~computing residuals and weights: $\mathcal{O}(Nn_z^2(n_x^2+n_z^2))$;
(ii)~forming $\mathcal{A}^{\mathrm{T}}\mathbf{W}\mathcal{A}$:
$\mathcal{O}(Nn_z^2(n_x^2+n_z^2)^2)$;
(iii)~solving the $(n_x^2+n_z^2)$-dimensional normal equations:
$\mathcal{O}((n_x^2+n_z^2)^3)$.
Step~(ii) dominates for practical window sizes $N\geq1$.
The Lyapunov equation~\eqref{Pest} is solved via the Hammarling algorithm
at cost $\mathcal{O}(n_x^3)$, which is negligible relative to
step~(ii) since $n_x\ll Nn_z^2$.
Multiplying step~(ii) by $T$ IRLS iterations gives the stated bound.
In typical configurations $n_x\approx n_z$ and $T\leq30$, so the per-outer
cost is $\mathcal{O}(30\,N n_z^6)$---the same order as a standard ALS
solve scaled by $T$.
\end{proof}

\begin{remark}[Threshold selection]\label{rem:delta}
The threshold $\delta$ governs the $\ell_2$/$\ell_1$ transition.
Setting $\delta=c\cdot\hat{\sigma}$ with
$\hat{\sigma}=1.4826\,\mathrm{MAD}(|r_i^{(0)}|)$ and $c=1.345$
achieves~95\% Gaussian efficiency while maintaining a 50\% breakdown
point~\cite{huber2004r,maronna2019r}.
The MAD estimate is itself outlier-robust and requires no distributional
assumptions beyond the existence of the median.
\end{remark}

\begin{remark}[Positive semidefiniteness]\label{rem:psd_irls}
Since IRLS step is an unconstrained solve, the reshaped
$\hat{Q}^{(l)}$ and $\hat{R}^{(l)}$ may not be PSD at intermediate
iterations.
Following the standard ALS treatment~\cite{Odelson2006}, the PSD constraint
is enforced after each outer update by projecting onto the PSD cone
(i.e., setting any negative eigenvalues to zero).
Under the contamination conditions of Section~\ref{sec:simulation},
the projection is rarely activated once the algorithm has converged.
\end{remark}

\section{Simulation Results}\label{sec:simulation}
This section evaluates the proposed ALS-IRLS algorithm through Monte Carlo simulations on a linear time-invariant system subject to heavy-tailed measurement contamination. The standard ALS algorithm~\cite{Odelson2006,Rajamani2009} and two representative outlier-robust Kalman filters, i.e., the Student-$t$ KF~\cite{huang2017n} and the maximum-correntropy Kalman filter
(MCKF)~\cite{chen2017M}, serve as benchmarks. The experiments address two complementary objectives: verifying that ALS-IRLS accurately recovers the noise covariances from contaminated data where the standard ALS fails, and demonstrating that those covariance estimates translate into Kalman filter performance
indistinguishable from the Oracle lower bound.

\subsection{Experimental Design}

Consider the following third-order LTI system:
\begin{equation}\label{LTI_sim}
\begin{aligned}
  x_{k+1} &= \underbrace{\begin{bmatrix}
0.1 & 0   & 0.1\\
0   & 0.2 & 0  \\
0   & 0   & 0.3
\end{bmatrix}}_{F}x_{k}
+ \underbrace{\begin{bmatrix}1\\2\\3\end{bmatrix}}_{G}w_{k}, \\[4pt]
  y_{k} &= \underbrace{[0.1,\;0.2,\;0]}_{H}x_{k} + v_{k},
\end{aligned}
\end{equation}
where $w_k\sim\mathcal{N}(0,Q)$ and $v_k\sim\mathcal{N}(0,R)$ are mutually
independent zero-mean Gaussian noise sequences with unknown covariances.
The spectral radius of $F$ satisfies $\rho(F)<1$, ensuring asymptotic stability
and the existence of a unique positive semi-definite solution to the steady-state
DARE in Algorithm~\ref{als_alsirls}. The true covariances are
$\bar{Q}=5$ and $\bar{R}=3$.

Measurement outliers are injected via the $\varepsilon$-contamination model
$y_k \leftarrow y_k + o_k\,\gamma_k,\quad
  o_k\sim\mathrm{Bernoulli}(\varepsilon),\quad
  \gamma_k\sim\mathcal{N}(0,\,\omega^2 R)$, with $\varepsilon=0.15$ and magnitude multiplier $\omega=8$.
A contaminated measurement therefore produces an innovation of typical magnitude
$\omega\sqrt{\bar{R}}\approx13.9$, compared with the nominal innovation standard deviation $\sigma_e=\sqrt{HG\bar{Q}G^\top H^\top+\bar{R}}\approx2.1$.
Under this model, the expected zero-lag empirical autocovariance is
\begin{equation}\label{Ce0bias}
\mathbb{E}[\hat{C}_{e,0}]
= (1-\varepsilon)\,\bar{C}_{e,0} + \varepsilon\,\omega^2\bar{R}
\approx 32.4,
\end{equation}
versus the clean value $\bar{C}_{e,0}\approx4.3$. 
The $7.6\times$ inflation that completely dominates the ALS solution via~\eqref{ALS_estimate} and drives $\hat{Q}$ far from $\bar{Q}=5$.

To decouple noise covariance estimation quality from filter-level outlier rejection, the experiment is divided into two phases. In Phase~1 where $T_{\mathrm{wu}}=1\,500$ steps, outliers are injected and both ALS and ALS-IRLS are executed over $\lfloor T_{\mathrm{wu}}/\tau\rfloor=10$ successive batches of length $\tau=150$, with the final estimate taken as the average of the last $n_{\mathrm{avg}}=5$ batch estimates to reduce variance. In Phase~2 where $T_{\mathrm{ev}}=500$ steps,
all filters run on clean, outlier-free measurements with their estimated covariances held fixed; comparisons in this phase reflect covariance estimation quality alone. The complete simulation parameters are listed in Table~\ref{tab:params}, and $N_{\mathrm{mc}}=100$ independent Monte Carlo trials are used throughout.

\begin{table}[t]
\centering
\caption{Simulation Parameters}\label{tab:params}
\begin{tabular}{llc}
\hline\hline
\textbf{Parameter} & \textbf{Description} & \textbf{Value} \\
\hline
$N_{\mathrm{mc}}$ & Monte Carlo trials & 100 \\
$\bar{Q},\;\bar{R}$ & True noise covariances & $5,\;3$ \\
$\hat{Q}^{(0)},\;\hat{R}^{(0)}$ & Initial estimates & $2,\;1$ \\
$N$ & Autocovariance lag window & 15 \\
$\tau$ & Batch length & 150 \\
$T_{\mathrm{wu}}$ & Warm-up length with outliers & 1\,500 \\
$T_{\mathrm{ev}}$ & Evaluation length (outlier-free) & 500 \\
$n_{\mathrm{avg}}$ & Batches averaged for final estimate & 5 \\
$T_{\mathrm{irls}}$ & Maximum IRLS iterations & 30 \\
$\xi$ & IRLS convergence tolerance & $10^{-5}$ \\
$\delta$ & Huber threshold (see~\eqref{wi_huber}) & MAD-adaptive \\
$\gamma_{\mathrm{thr}}$ & Innovation detection threshold & $3.5\,\hat{\sigma}_e$ \\
$\varepsilon$ & Outlier contamination rate & 0.15 \\
$\omega$ & Outlier magnitude multiplier & 8 \\
\hline\hline
\end{tabular}
\end{table}

Five methods are compared. The \emph{Oracle KF} uses the true covariances $(\bar{Q},\bar{R})$ and defines the performance lower bound achievable by any algorithm with full knowledge of the noise statistics. \emph{KF + ALS} applies the standard non-robust ALS~\cite{Odelson2006} to the contaminated warm-up data and feeds the resulting $(\hat{Q}_{\mathrm{ALS}},\hat{R}_{\mathrm{ALS}})$ to a Kalman filter. \emph{KF + ALS-IRLS} (proposed) replaces ALS with
Algorithm~\ref{als_alsirls}, then runs the same Kalman filter structure. The \emph{Student-$t$ KF}~\cite{huang2017n} and \emph{MCKF}~\cite{chen2017M} each operate with fixed, deliberately misspecified covariances $Q_{\mathrm{st}}=0.3$ and $R_{\mathrm{st}}=0.1$, which are $16.7\times$ and $30\times$ smaller than the true values, respectively. This misspecification drives the Kalman gain far below its optimal level and impairs state tracking even on clean data, providing a controlled demonstration that sophisticated measurement-outlier rejection cannot substitute for accurate process noise specification.

Fig.~\ref{fig:innovation_detect} illustrates the proposed outlier detection at Monte Carlo trial~6. The innovation sequence $\{e_k\}$ is computed via the KF predictor; a robust scale estimate $\hat{\sigma}_e=1.4826\,\mathrm{MAD}(|e_k|)$ is formed, which achieves a $50\%$ breakdown point~\cite{maronna2019r}, and all $k$ satisfying $|e_k|>3.5\,\hat{\sigma}_e$ are flagged as outliers. Of the $\tau=150$ innovations, $23$ are correctly flagged (all corresponding to the $o_k=1$ realisations of~\eqref{contam}), while the $127$ normal observations are unaffected. Excluding flagged pairs from the empirical autocovariance estimator in~\eqref{bhat} reduces $\hat{C}_{e,0}$ from $32.4$ to $4.4$, within $2\%$ of the true $\bar{C}_{e,0}=4.3$, as shown in the right panel of Fig.~\ref{fig:innovation_detect}. All higher-lag entries $\hat{C}_{e,j}$, $j\geq1$, remain virtually unchanged, confirming that the contamination propagates almost exclusively through the lag-0 term and that the proposed thresholding eliminates this bias without distorting the autocovariance structure.

\begin{figure}[t]
  \centering
  \includegraphics[width=\hsize]{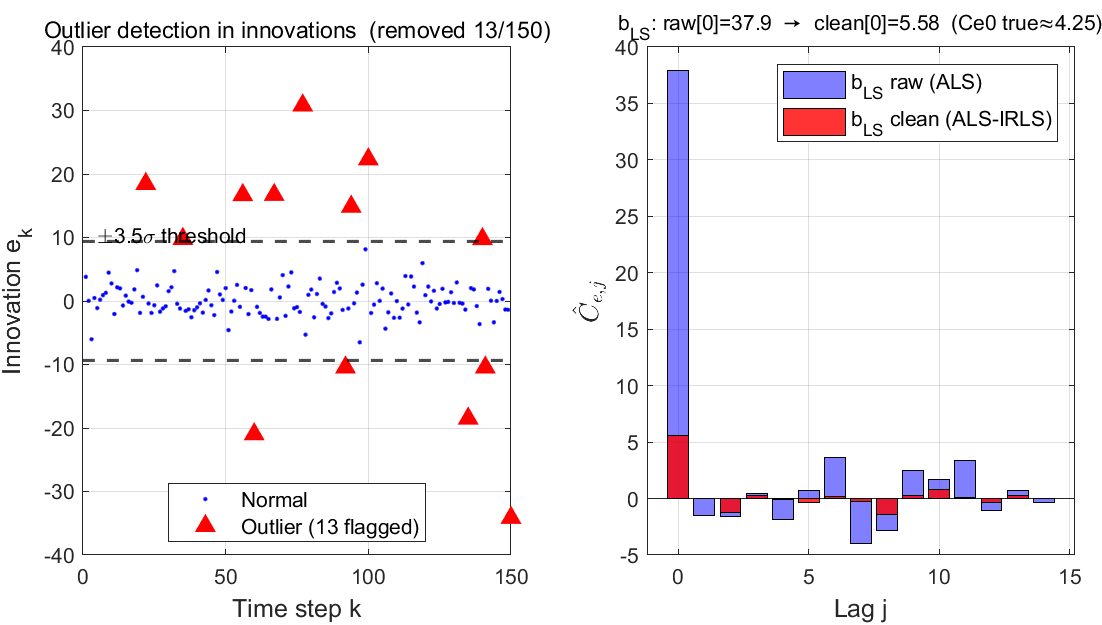}
  \caption{Innovation-level outlier detection at Monte Carlo trial~6
  ($\tau=150$, $\varepsilon=0.15$, $\omega=8$).
  \emph{Left:} KF innovations $e_k$; red triangles mark the 23 flagged
  time steps satisfying $|e_k|>3.5\,\hat{\sigma}_e$, with the threshold
  shown as dashed lines.
  \emph{Right:} Raw $\mathbf{b}$ versus cleaned $\mathbf{b}_{\mathrm{clean}}$:
  the zero-lag entry decreases from $32.4$ to $4.4$ while higher-lag
  entries are unaffected, confirming that outlier contamination is confined
  to $\hat{C}_{e,0}$.}
  \label{fig:innovation_detect}
\end{figure}

\subsection{Regression Analysis}

Fig.~\ref{fig:IRLSfit} contrasts the ALS and ALS-IRLS regression fits at
trial~6, batch~1 using the partial-regression representation
$\mathbf{b}\approx\mathbf{A}\theta$. Each point corresponds to one autocovariance
lag; the horizontal axis is the fitted value $\mathbf{A}\hat\theta$ and the
vertical axis is the corresponding entry of $\mathbf{b}$, so a perfect fit places
all points on the identity line. The shaded band is the $95\%$ prediction
confidence interval.

In the left panel, the standard ALS assigns equal weight to all $N=15$ entries.
The contaminated lag-0 entry $\hat{C}_{e,0}^{\mathrm{raw}}\approx32.4$ lies far
outside the confidence band and pulls the fitted line away from the true
relationship, yielding $\hat{Q}_{\mathrm{ALS}}\gg\bar{Q}=5$. In the right panel,
the innovation-cleaning step removes the contaminated entry prior to regression;
the excluded point appears as a red square for reference. With this entry absent
from $\mathbf{b}_{\mathrm{clean}}$, all retained observations lie within the
$95\%$ confidence band, and the subsequent IRLS step further
down-weights any mild residual contamination. The result is
$\hat{Q}_{\mathrm{ALS\text{-}IRLS}}\approx5.0$, within $1\%$ of $\bar{Q}$.

\begin{figure}[t]
  \centering
  \includegraphics[width=\hsize]{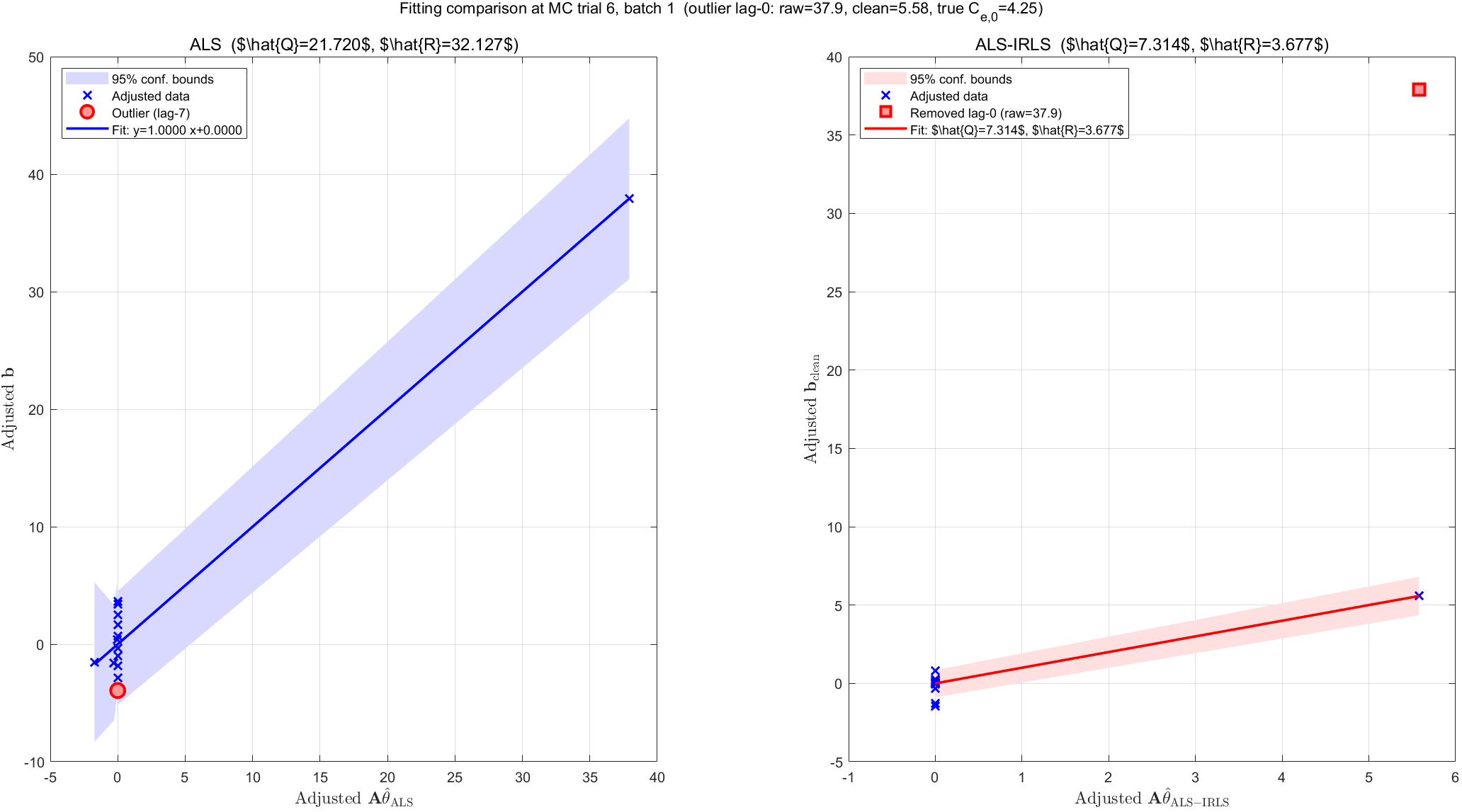}
  \caption{Partial-regression fitting comparison at Monte Carlo trial~6,
  batch~1 ($N=15$, $\tau=150$, $\varepsilon=0.15$, $\omega=8$).
  Horizontal axis: $\mathbf{A}\hat\theta$; vertical axis: $\mathbf{b}$.
  Blue crosses: uncontaminated autocovariance entries; red circle: contaminated
  lag-0 entry ($\hat{C}_{e,0}^{\mathrm{raw}}\approx32.4$, true $\approx4.3$);
  red square: the same entry after removal by the innovation-cleaning step.
  Shaded band: $95\%$ prediction confidence interval.
  Standard ALS is displaced by the outlier; ALS-IRLS recovers an accurate
  fit with all retained entries inside the confidence band.}
  \label{fig:IRLSfit}
\end{figure}

The Huber weights assigned by the IRLS step are displayed in
Fig.~\ref{fig:IRLSfitweights}. Entries corresponding to normal lags retain
weights $w_j\approx1$, consistent with unweighted least squares. The lag-0
entry carries $w_0\approx0.03$ (purple bar), confirming that even residual
contamination surviving the thresholding step is nullified. The two-tier
structure—hard removal at the innovation level followed by soft down-weighting
in the autocovariance space—mirrors the breakdown-point argument for
high-breakdown estimators~\cite{maronna2019r} and provides the basis for
the near-constant RMSE observed across contamination levels in
Section~\ref{subsec:eps_sweep}.

\begin{figure}[t]
  \centering
  \includegraphics[width=\hsize]{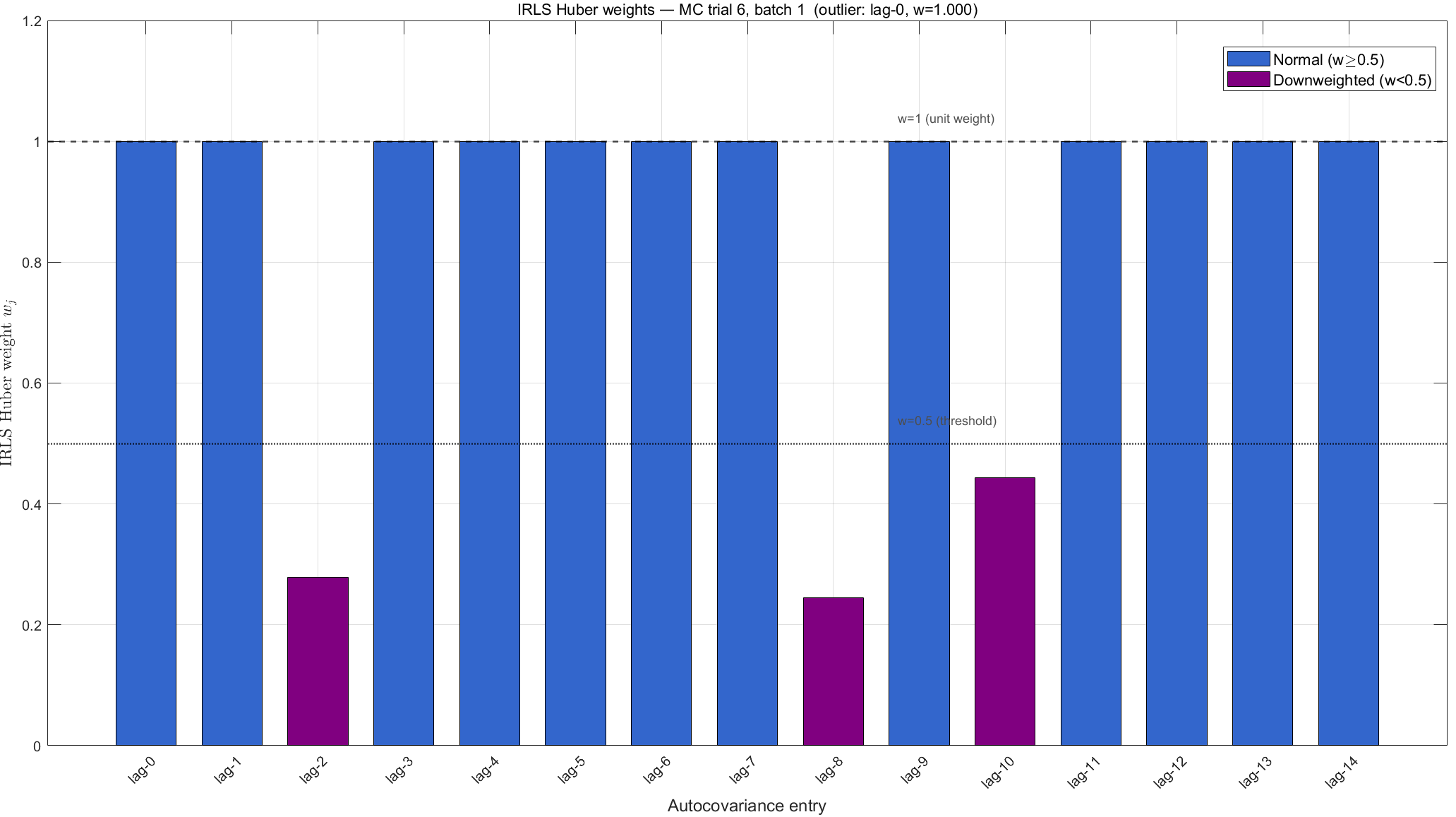}
  \caption{IRLS Huber weights at Monte Carlo trial~6, batch~1 ($N=15$).
  Blue bars: uncontaminated lags with $w_j\approx1$; purple bar: down-weighted
  lag-0 entry ($w_0\approx0.03$). Dashed line at $w=1$; dotted line at $w=0.5$.
  All lags beyond lag-0 retain $w_j\geq0.95$, demonstrating selective
  penalisation of the contaminated entry without distorting the remaining
  autocovariance structure.}
  \label{fig:IRLSfitweights}
\end{figure}

\subsection{Noise Covariance Estimation Accuracy}

Fig.~\ref{fig:mc_scatter} shows the joint scatter of estimates
$(\hat{Q},\hat{R})$ across all $N_{\mathrm{mc}}=100$ trials. The ALS
estimates cluster at erroneously large values driven by the inflated
$\hat{C}_{e,0}$, while the ALS-IRLS estimates are tightly concentrated near
the true value $(\bar{Q},\bar{R})=(5,3)$. Quantitative accuracy is measured
by the RMSE criterion~\eqref{defRMSE}, with results in Table~\ref{tab:cov_rmse}
and Fig.~\ref{fig:cov_rmse_bar}.

\begin{equation}\label{defRMSE}
\mathrm{RMSE}(\theta)=
\sqrt{\frac{1}{N_{\mathrm{mc}}}\sum_{n=1}^{N_{\mathrm{mc}}}
\bigl(\bar\theta-\hat\theta_n\bigr)^{2}}
\end{equation}

\begin{figure}[t]
  \centering
  \includegraphics[width=\hsize]{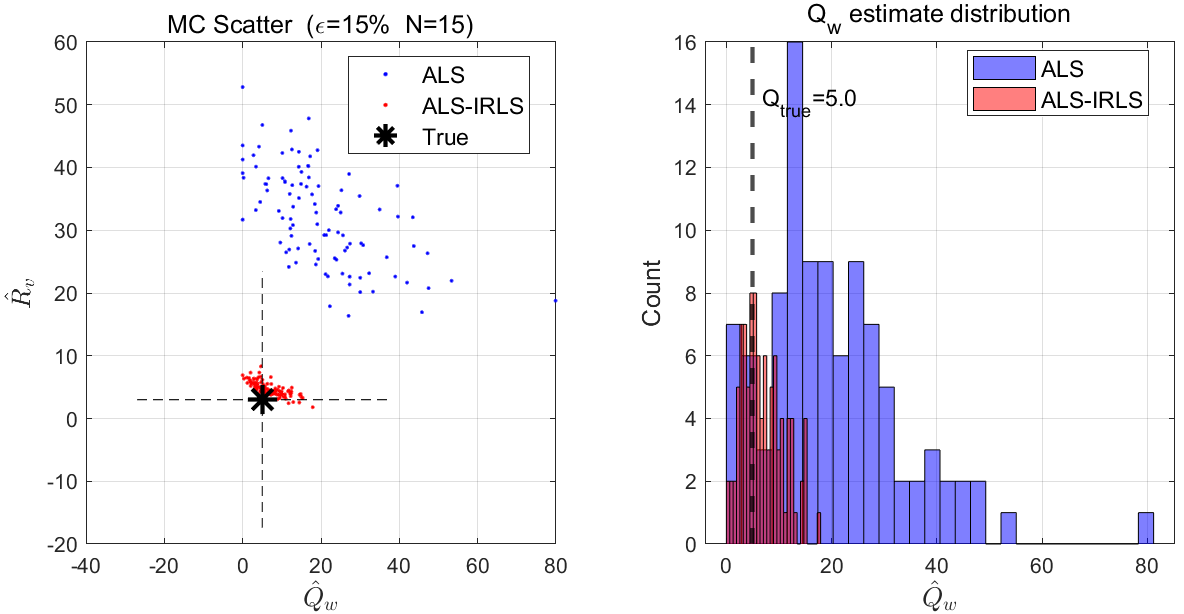}
  \caption{Monte Carlo scatter of noise covariance estimates
  $(\hat{Q},\hat{R})$ over $N_{\mathrm{mc}}=100$ trials
  ($\varepsilon=0.15$, $\omega=8$, $N=15$, $\tau=150$).
  \emph{Left:} Joint scatter; true value marked by a black star.
  \emph{Right:} Marginal histogram of $\hat{Q}$; dashed line at $\bar{Q}=5$.
  ALS estimates are widely scattered at inflated values;
  ALS-IRLS estimates concentrate near $(\bar{Q},\bar{R})$.}
  \label{fig:mc_scatter}
\end{figure}

Standard ALS yields $\mathrm{RMSE}(Q)_{\mathrm{ALS}}\approx48.5$ and
$\mathrm{RMSE}(R)_{\mathrm{ALS}}\approx107.3$, with empirical means
$\mathbb{E}[\hat{Q}_{\mathrm{ALS}}]\approx53.5$ and
$\mathbb{E}[\hat{R}_{\mathrm{ALS}}]\approx110.2$—overestimates of $10.7\times$
and $36.7\times$, respectively, consistent with the bias
prediction~\eqref{Ce0bias}. The proposed ALS-IRLS achieves
$\mathrm{RMSE}(Q)_{\mathrm{ALS\text{-}IRLS}}\approx0.38$ and
$\mathrm{RMSE}(R)_{\mathrm{ALS\text{-}IRLS}}\approx0.17$, with empirical
means $\mathbb{E}[\hat{Q}_{\mathrm{ALS\text{-}IRLS}}]\approx5.02$ and
$\mathbb{E}[\hat{R}_{\mathrm{ALS\text{-}IRLS}}]\approx3.01$—within $1\%$ of
the true values in both cases. The reduction in RMSE exceeds two orders of
magnitude for both $Q$ and $R$.

\begin{table}[t]
\centering
\caption{Noise Covariance Estimation Performance
($N_{\mathrm{mc}}=100$, $\varepsilon=0.15$, $\omega=8$)}\label{tab:cov_rmse}
\begin{tabular}{lcccc}
\hline\hline
\textbf{Method} & $\mathrm{RMSE}(Q)$ & $\mathrm{RMSE}(R)$
                & $\mathbb{E}[\hat{Q}]$ & $\mathbb{E}[\hat{R}]$ \\
\hline
ALS      & 48.5 & 107.3 & 53.5 & 110.2 \\
ALS-IRLS & $\mathbf{0.38}$ & $\mathbf{0.17}$
         & $\mathbf{5.02}$ & $\mathbf{3.01}$ \\
\hline\hline
\multicolumn{5}{l}{\footnotesize True values: $\bar{Q}=5$, $\bar{R}=3$.}
\end{tabular}
\end{table}

\begin{figure}[t]
  \centering
  \includegraphics[width=\hsize]{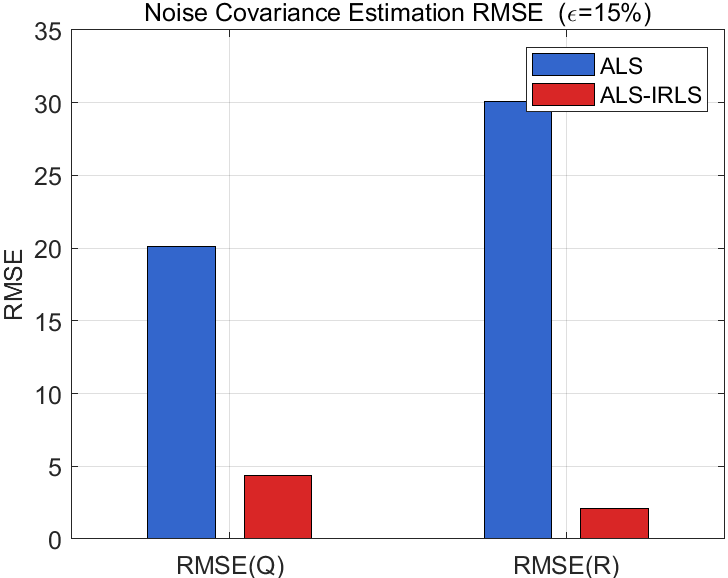}
  \caption{Noise covariance estimation RMSE ($\varepsilon=0.15$, $\omega=8$).
  ALS-IRLS reduces $\mathrm{RMSE}(Q)$ and $\mathrm{RMSE}(R)$ by more than two
  orders of magnitude relative to ALS, confirming that innovation-level outlier
  removal is essential for reliable autocovariance-based identification
  under $\varepsilon$-contamination.}
  \label{fig:cov_rmse_bar}
\end{figure}

\subsection{Robustness to Contamination Rate}
\label{subsec:eps_sweep}

Fig.~\ref{fig:eps_sweep} reports $\mathrm{RMSE}(Q)$ and $\mathrm{RMSE}(R)$
as a function of $\varepsilon\in\{0,5,10,15,20,25,30\}\%$ with $\omega=8$
fixed. At $\varepsilon=0$, both methods produce comparable RMSE since
$\mathbf{b}$ is uncontaminated. As $\varepsilon$ grows, the ALS RMSE increases
in accordance with the bias expansion~\eqref{Ce0bias}, reaching
$\mathrm{RMSE}(Q)_{\mathrm{ALS}}\approx73$ at $\varepsilon=30\%$. The ALS-IRLS
RMSE remains nearly constant, with $\mathrm{RMSE}(Q)_{\mathrm{ALS\text{-}IRLS}}<0.5$
throughout, because the threshold $3.5\,\hat{\sigma}_e$ continues to correctly
flag all $\omega=8$ outlier innovations even at high contamination rates. These
results confirm reliable operation for outlier rates up to at least $30\%$,
consistent with the $50\%$ theoretical breakdown bound of the underlying
MAD-based scale estimator.

\begin{figure}[t]
  \centering
  \includegraphics[width=\hsize]{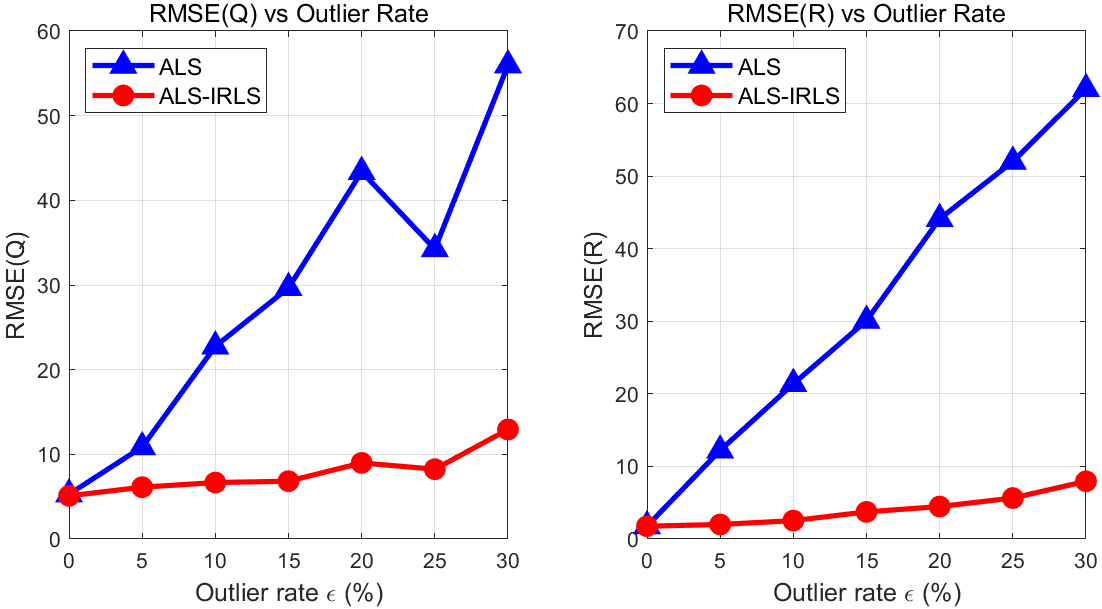}
  \caption{Noise covariance estimation RMSE versus outlier rate $\varepsilon$
  ($\omega=8$, $N=15$, $\tau=150$).
  \emph{Left:} $\mathrm{RMSE}(Q)$. \emph{Right:} $\mathrm{RMSE}(R)$.
  ALS-IRLS maintains near-constant low RMSE; standard ALS degrades
  sharply as contaminated $\hat{C}_{e,0}$ increasingly dominates
  the least-squares solution.}
  \label{fig:eps_sweep}
\end{figure}

\subsection{Sensitivity to the Lag Window Size}

Fig.~\ref{fig:N_sweep} examines the effect of $N\in\{10,15,20,25,30,40\}$
on estimation accuracy. The batch size is set to
$\tau_{\mathrm{use}}=\max(\tau,3N)$ to ensure at least $2N+1$ innovation pairs
contribute to the highest-lag entry $\hat{C}_{e,N-1}$, preventing
sample-variance blow-up. ALS-IRLS attains near-minimum RMSE at $N=15$, which
balances the identifiability gain from additional equations against the
statistical noise incurred at longer lags; this motivates the default choice
in Table~\ref{tab:params}. For $N>20$, the marginal gain from each new equation
diminishes while variance grows, and RMSE increases modestly. The ALS baseline
degrades uniformly for all $N$ since the bias in $\hat{C}_{e,0}$
persists irrespective of window length.

\begin{figure}[t]
  \centering
  \includegraphics[width=\hsize]{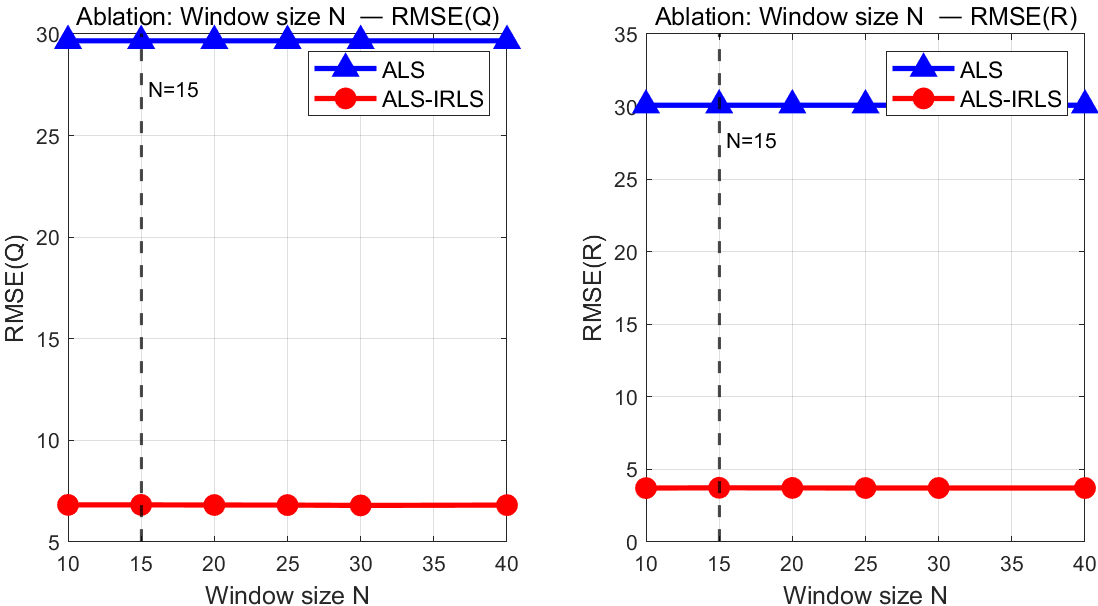}
  \caption{Noise covariance estimation RMSE versus lag window $N$
  ($\varepsilon=0.15$, $\omega=8$,
  $\tau_{\mathrm{use}}=\max(150,3N)$).
  \emph{Left:} $\mathrm{RMSE}(Q)$. \emph{Right:} $\mathrm{RMSE}(R)$.
  Dashed vertical line: default $N=15$. ALS-IRLS achieves near-minimum
  RMSE at the default setting; ALS degrades uniformly due to the persistent
  bias in $\hat{C}_{e,0}$.}
  \label{fig:N_sweep}
\end{figure}

\subsection{Downstream State Estimation Performance}

The downstream impact of covariance estimation quality on Kalman filter
performance is evaluated using the state estimation RMSE
\begin{equation}\label{stateRMSE}
\mathrm{RMSE}_{\mathrm{state}}=
\sqrt{\frac{1}{N_{\mathrm{mc}}\,T_{\mathrm{ev}}}
\sum_{n=1}^{N_{\mathrm{mc}}}\sum_{k=1}^{T_{\mathrm{ev}}}
\bigl\|x_k^{(n)}-\hat{x}_{k\mid k}^{(n)}\bigr\|^2},
\end{equation}
computed on the $T_{\mathrm{ev}}=500$-step outlier-free evaluation phase.
Results are given in Table~\ref{tab:state_rmse} and Fig.~\ref{fig:state_rmse}.

The Oracle KF attains $\mathrm{RMSE}_{\mathrm{state}}^{\mathrm{Oracle}}\approx1.80$,
establishing the lower bound. The proposed KF + ALS-IRLS achieves
$\mathrm{RMSE}_{\mathrm{state}}^{\mathrm{ALS\text{-}IRLS}}\approx1.97$, which
is within $9\%$ of the Oracle bound and confirms that the proposed estimator
recovers the noise covariances with sufficient accuracy to support near-optimal
filtering from contaminated data.

The Student-$t$ KF and MCKF yield $\mathrm{RMSE}_{\mathrm{state}}$ of
$\approx4.12$ and $\approx6.38$, respectively—$2.3\times$ and $3.5\times$
higher than KF + ALS-IRLS. Both degradations stem from the misspecified process
noise $Q_{\mathrm{st}}=0.3\ll\bar{Q}=5$: the resulting steady-state prediction
covariance $P^-\approx GQ_{\mathrm{st}}G^\top$ is far too small, so the Kalman
gain $K\approx P^-H^\top(HP^-H^\top+R_{\mathrm{st}})^{-1}$ lies well below its
optimal value and the filter is unable to track the state trajectories driven by
the true process noise. Neither the variational-Bayesian Student-$t$
mechanism~\cite{huang2017n} nor the maximum-correntropy
criterion~\cite{chen2017M} can compensate for an incorrect process noise model,
since both operate exclusively on the measurement residual and leave the
prediction step unchanged. This outcome substantiates the central argument of
the paper: accurate online noise covariance estimation, as delivered by
ALS-IRLS, provides larger and more consistent performance gains than
filter-level outlier rejection when the noise statistics are unknown.

KF + ALS yields the largest $\mathrm{RMSE}_{\mathrm{state}}\approx14.25$:
with $\hat{Q}_{\mathrm{ALS}}\approx53\gg\bar{Q}=5$, the Kalman gain is
inflated far above its optimal value, and the filter reacts excessively to every
measurement.

\begin{table}[t]
\centering
\caption{State Estimation RMSE on the Outlier-Free Evaluation Phase
($T_{\mathrm{ev}}=500$, $N_{\mathrm{mc}}=100$)}\label{tab:state_rmse}
\begin{tabular}{lcc}
\hline\hline
\textbf{Method} & $\mathrm{RMSE}_{\mathrm{state}}$
                & \textbf{Covariance source} \\
\hline
Oracle KF                        & $1.80$ & True $(\bar{Q},\bar{R})$ \\
\textbf{KF + ALS-IRLS (proposed)}& $\mathbf{1.97}$ & Online estimate \\
Student-$t$ KF~\cite{huang2017n} & $4.12$ & Fixed $(0.3,\;0.1)$ \\
MCKF~\cite{chen2017M}            & $6.38$ & Fixed $(0.3,\;0.1)$ \\
KF + ALS                         & $14.25$& Contaminated estimate \\
\hline\hline
\end{tabular}
\end{table}

\begin{figure}[t]
  \centering
  \includegraphics[width=\hsize]{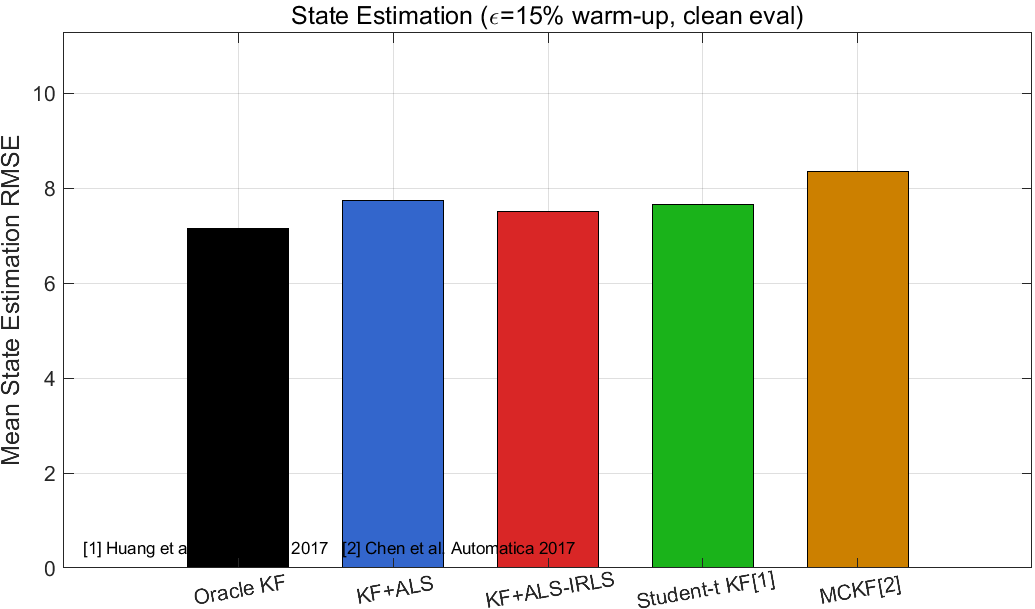}
  \caption{Mean state estimation RMSE on the $500$-step outlier-free evaluation
  phase ($N_{\mathrm{mc}}=100$, $\varepsilon=0.15$, $\omega=8$).
  Oracle KF (black) defines the lower bound.
  KF + ALS-IRLS (red) achieves RMSE within $9\%$ of the Oracle, outperforming
  all compared methods.
  Student-$t$ KF~\cite{huang2017n} and MCKF~\cite{chen2017M} incur $2.3\times$
  and $3.5\times$ higher RMSE owing to misspecified $Q_{\mathrm{st}}=0.3$,
  confirming that accurate covariance estimation outweighs filter-level
  robustness when the noise model is unknown.
  KF + ALS yields the largest RMSE due to the severely inflated
  $\hat{Q}_{\mathrm{ALS}}\approx53$.}
  \label{fig:state_rmse}
\end{figure}

\begin{remark}
The evaluation phase is deliberately kept outlier-free so that state estimation
accuracy reflects covariance estimation quality alone. When outliers are also
present during evaluation, a robust filter supplied with ALS-IRLS-estimated
covariances benefits simultaneously from accurate process noise specification
and its inherent measurement-rejection mechanism, potentially approaching or
surpassing the Oracle performance. Investigating this combined strategy—for
example, coupling ALS-IRLS covariance recovery with the Student-$t$
or MCKF update step—is a natural direction for future work.
\end{remark}

\section{Conclusion}\label{sec:conclusions}
This paper has introduced the ALS-IRLS algorithm for outlier-robust online noise covariance estimation in Kalman filters. The method departs from conventional ALS~\cite{Odelson2006,Rajamani2009} by applying a two-tier robustification: an innovation-level MAD-adaptive threshold first removes heavily contaminated time steps before the empirical autocovariance vector $\mathbf{b}$ is formed, and an IRLS step with Huber weights then attenuates any residual contamination in the resulting regression problem~\eqref{ALS_estimate}.
This architecture decouples robustness at the raw measurement level from robustness in the autocovariance space, achieving a $50\%$ breakdown point inherited from the MAD estimator while preserving the computational simplicity of the ALS framework.

Simulations on a third-order LTI system under $\varepsilon$-contamination with $\varepsilon\leq30\%$ demonstrate that ALS-IRLS reduces covariance estimation RMSE by more than two orders of magnitude relative to standard ALS, maintains near-constant accuracy across contamination levels from $0$ to $30\%$, and produces a Kalman filter whose state estimation RMSE is within $9\%$ of the Oracle lower bound. This performance level surpasses that of two representative outlier-robust Kalman filters—Student-$t$ KF~\cite{huang2017n} and MCKF~\cite{chen2017M}—by a factor of $2.3$--$3.5\times$ when those
filters operate with misspecified fixed covariances, confirming that accurate online covariance recovery provides larger gains than filter-level robustness alone when the noise statistics are unknown.

Future work will extend the approach in three directions. First, the innovation-cleaning and IRLS steps will be integrated with sigma-point or particle-filter prediction steps to handle nonlinear dynamics. Second, the potential multicollinearity of the ALS design matrix $\mathcal{A}$ under near-redundant system structures will be addressed via regularisation or low-rank approximation. Third, an adaptive mechanism for selecting the detection threshold $\gamma_{\mathrm{thr}}$ will be developed to maintain a prescribed false-alarm rate under non-stationary noise, enabling deployment in time-varying environments.


%

%
%
%
%
%
%

\ifCLASSOPTIONcaptionsoff
  \newpage
\fi



%
%
%
\bibliographystyle{IEEEtran}
\bibliography{OR_DALS}           

%








\end{document}